\documentclass[12pt,twoside]{article}
\usepackage[T1]{fontenc}
\usepackage[utf8]{inputenc}
\usepackage[a4paper]{geometry}
\usepackage{color}
\usepackage{mathrsfs}
\usepackage{amsmath}
\usepackage{amsthm}
\usepackage{amssymb}
\usepackage{setspace}

\usepackage{xcolor}
\usepackage{mathtools}
\DeclarePairedDelimiter{\abs}{\lvert}{\rvert}
\usepackage{enumitem}
\usepackage{amsfonts}

\usepackage{enumitem}
\usepackage{bigints}

\allowdisplaybreaks[4]

\makeatletter
\numberwithin{equation}{section}
\numberwithin{figure}{section}

\theoremstyle{plain}
\newtheorem{thm}{\protect\theoremname}[section]
\theoremstyle{plain}
\newtheorem{cor}[thm]{\protect\corollaryname}
\theoremstyle{plain}
\newtheorem{lem}[thm]{\protect\lemmaname}
\theoremstyle{definition}
\newtheorem{defn}[thm]{\protect\definitionname}
\theoremstyle{plain}
\newtheorem{prop}[thm]{\protect\propositionname}
\newcommand{\lyxaddress}[1]{
	\par {\raggedright #1
	\vspace{1.4em}
	\noindent\par}
}

\usepackage[english]{babel}
\usepackage{latexsym}
\usepackage{amsthm}
\usepackage{eucal}
\usepackage{eufrak}
\usepackage{epstopdf}
\usepackage{graphicx}

\theoremstyle{plain}

\newtheoremstyle{boldremark}
    {\dimexpr\topsep/2\relax} 
    {\dimexpr\topsep/2\relax} 
    {}          
    {}          
    {\bfseries} 
    {.}         
    {.5em}      
    {}          

\theoremstyle{boldremark}

\usepackage{fancyhdr}
\usepackage{blindtext}
\usepackage{titlesec}

\titleformat{name=\chapter}[display]
{\normalfont\Huge\itshape}
{\titlerule[1pt]\vspace{-33pt}\filleft%
  \parbox[t]{6em}{%
    \raggedleft%
    \rule{\linewidth}{0.5ex}\newline%
    \chaptertitlename\ \thechapter%
  }%
}
{4pc}
{\normalfont\upshape\bfseries\Huge}
\makeatother

\providecommand{\corollaryname}{Corollary}
\providecommand{\definitionname}{Definition}
\providecommand{\lemmaname}{Lemma}
\providecommand{\propositionname}{Proposition}
\providecommand{\theoremname}{Theorem}

\begin{document}

\begin{singlespace}

\title{\noindent \textbf{On widely degenerate \textit{p}-Laplace equations with symmetric data }}
\end{singlespace}
\begin{singlespace}

\author{\noindent Stefania Russo}
\end{singlespace}
\begin{singlespace}

\date{\noindent }
\end{singlespace}
\maketitle
\begin{abstract}
\begin{singlespace}
We consider the Dirichlet problems 
\begin{equation*}\label{1}
    \begin{cases}
   - \mathrm{div} \Bigg( \,\Big( \abs{\nabla u_{p}} -1 \Big)_{+}^{p-1} \displaystyle{ \frac{\nabla u_{p}}{\abs{\nabla u_{p}}} } \Bigg) = f    \qquad \text{                    in  } B_R \qquad\\
       u_{p}=0 \, \hspace{14em} \text{                                on  } \partial{B_R},
      \end{cases}
      \end{equation*}
where $p > 1$ and $B_R \subseteq \mathbb{R}^N, \, N\geq 2$, is the open ball centered at the origin with radius $R>0$ .
\\
Through a well-known result by Talenti \cite{Tal}, we explicitly express the gradient of the solution $u_p$ outside the set $\{  \abs{\nabla u_p}\leq 1\}$, if the datum $f$ is a non-negative integrable radially decreasing function. This allows us to establish some sharp higher regularity results for the weak solutions, assuming that the datum $f$ belongs to a suitable Lorentz space, i.e. under a weaker assumption on the datum with respect to the available literature.  Moreover we analyze the behaviour of $u_p$ as $p \to 1^+$.
\end{singlespace}
\end{abstract}

\begin{singlespace}
\end{singlespace}

\begin{singlespace}
\end{singlespace}

\begin{singlespace}

\section{Introduction}
\end{singlespace}
Let us fix a ball $B_R \subset \mathbb{R}^N \text{ with } \,N \geq 2$ and consider the following family of Dirichlet problems 
\begin{equation}\label{1}
    \begin{cases}
   - \mathrm{div} \Bigg( \,\Big( \abs{\nabla u_{p}} -1 \Big)_{+}^{p-1} \displaystyle{ \frac{\nabla u_{p}}{\abs{\nabla u_{p}}} } \Bigg) = f    \qquad \text{                    in  } B_R \qquad\\
       u_{p}=0 \, \hspace{14em} \text{                                on  } \partial{B_R},
      \end{cases}
      \end{equation}
where $p > 1$ and $( \,\, \cdot \,\, )_+$ stands for the positive part.\\
The main feature of the equation in the Dirichlet problem at \eqref{1} is that it is widely degenerate, i.e. it behaves as the \textit{p}-Laplace equation only for large values of the modulus of the gradient of the solution and therefore fits into the wider context of the asymptotically regular problems (see [12-17],
\cite{I9,I10}).\\
 This kind of equations attracted a great interest in the last few years since they naturally arise in optimal transport problems with congestion effects (for a detailed explanation of this connection, we refer to \cite{I12,Bra}).\\ 
When dealing with non-degenerate \textit{p}-Laplace equations, the qualitative properties of the solutions can be obtained by means of the so-called symmetrization technique. More precisely, starting from the pioneering paper by Talenti (\cite{Tal}), it has been clear that the properties of the weak solutions to different kind of partial differential equations  can be derived by a comparison with the corresponding symmetrized problem. For an exhaustive list of applications of these techniques and references on this subject, we refer to  \cite{Barbato} (see, for example,  \cite{Aniso} for anisotropic elliptic operators, \cite{Parabolic} for the parabolic case, \cite{High1,High2} for higher-order operators). \\
A first important step in this comparison procedure is to analyze the qualitative properties of the solution to a problem with suitable symmetries.\\
As far as we know, such technique hasn't been exploited yet in case of widely degenerate problems and the aim of this paper is to give a first contribution by analyzing the qualitative properties of the weak solutions to \eqref{1} in case the datum $f$ is a non-negative radially decreasing function.\\
Although the celebrated result of Talenti can be applied to a wide class of elliptic problems, it cannot be directly applied to the equation in \eqref{1}. Indeed, the expression of the solution to a problem with a symmetric right-hand side of the form 
    $$ - \mathrm{div}\Big( a \left(  \nabla u_p \right)  \Big) = \, f,$$
where $ A(\abs{\xi}) \leq \langle a(\xi), \xi \rangle$, heavily relies on the ability to invert the function $\frac{A(r)}{r}$, which in our case is reduced to $(r-1)^{p-1}_+$, clearly invertible only for $r \geq 1$. Hence, the application of Talenti's Theorem, requires the introduction of a family of approximating problems that are uniformly elliptic.\\
Indeed, with this argument and a limiting procedure, we are able to give the explicit expression of the gradient of the solution to \eqref{1} a.e. outside the set $\{  \abs{\nabla u_p}\leq 1\}$. More precisely, our main result is the following
\begin{thm}\label{Teo1}
    Let $u_p$ be a weak solution of \eqref{1} with $f \in L^{N,\infty}(B_R)$
    , radially symmetric and decreasing.
    Then for every $p >1$, it holds the following
    \begin{equation*} \label{MIAsolu}
       \frac{( \abs {\nabla u_p}-1)_+ }{\abs{ \nabla u_p}} \nabla u_p =  -  \left( \frac{\abs{x}}{N}f^{**}(C_N \abs{x}^N)\right)^{\frac{1}{p-1}} \, \frac{x}{\abs{x}} \qquad  \text{a.e. in  }\, B_R
\end{equation*}
and so
 \begin{equation}\label{Dupexplicit}
        \nabla u_p =- \left[\,  1 \,+ \, \left( \frac{\abs{x}}{N}f^{**}(C_N \abs{x}^N)\right)^{\frac{1}{p-1}} \,\right]\, \frac{x}{\abs{x}} \, \qquad  
    \end{equation}
    a.e. in the set  $\{  \abs{ \nabla u_p} > 1\}$.
\end{thm}
\noindent As a consequence, we get
  \begin{equation} \label{MIAsolu}
        u_{p} (x) = \frac{1}{N C^{\frac{1}{N}}_{N}}\, \displaystyle{\int\limits^{C_N R^N}_{C_N \abs{x}^N} } \left( 1 + \left( \frac{s^{ \frac{1}{N}}}{N C^{\frac{1}{N}}_{N}}   f^{**}(s)  \right)^{\frac{1}{p-1}} \right) s^{-1+ \frac{1}{N}} \, ds.
\end{equation}
a.e. in the set $\{  \abs{ \nabla u_p} > 1\}$.\\
In order to prove the previous Theorem, we argue by approximation introducing a suitable family of uniformly elliptic problems. For such problems we are legitimate to use the result by Talenti (\cite{Tal}) that, under our assumption on $f$, allows us to write explicitely their solutions $u_{\varepsilon, p }$ for $p\geq2$. 
 Next we demonstrate, through a direct calculation, that this expression represents the solution also for $1<p<2$. After this, we show that a suitable function of the gradient of these solutions, which vanishes in the degeneracy set of \eqref{1}, strongly converges to the same function calculated along the gradient of  the solution to \eqref{1}. More precisely we show that
\begin{align}
( \abs{\nabla u^{\varepsilon}_p} -1)_{+} \, \frac{\nabla u^{\varepsilon}_p}{\abs{\nabla u^{\varepsilon}_p}} \, \rightarrow \, ( \abs{\nabla u_p} -1 )_{+} \, \frac{\nabla u_p}{\abs{\nabla u_p}}  \quad \mathrm{   a.e. \, in  }\, B_R. \notag
\end{align}
 At this point we are able to express $\nabla u_p$ outside the unit ball and this expression can be extended to the whole $B_R$. \\Actually, it is well know that widely degenerate problems lose the uniqueness of their solution and here, through a direct calculation, we can choose the function at \eqref{MIAsolu} as a solution to \eqref{1} in the whole $B_R$. As a matter of fact, it is possible to have infinitely many different solutions because the operator vanishes in the unit ball. Since every solution works, we opt for the one we found previously.\\
It is well known that for solutions to the equation in \eqref{1}, no more than Lipschitz regularity can be expected, even in the case $f=0$. In fact, every Lipschitz continuous function with  Lipschitz constant less than or equal to 1 is a solution to
\begin{equation*}
    - \mathrm{div} \Bigg( \,\Big( \abs{\nabla u_{p}} -1 \Big)_{+}^{p-1} \displaystyle{ \frac{\nabla u_{p}}{\abs{\nabla u_{p}}} } \Bigg) = 0
\end{equation*}
Also, we would like to mention that even when the datum is different from zero, the Lipschitz regularity of the solution holds if the datum belongs to a space smaller than $L^{N}$  \cite{I13, BraLinf, Bra}.\\
Here, as a consequence of Theorem \ref{Teo1}, we immediately obtain the higher integrability for the gradient of the weak solution and its boundedness under weaker assumption of $f$ with respect to the existing literature \cite{I12, Bra}. Indeed, the following holds
\begin{cor}\label{Cor1}
Let $u_p \in W_0^{1,p}(B_R)$ be a solution of \eqref{1}. If $f \in L^{r, \infty}(B_R)$ radially decreasing, then
$$ \nabla u_p \in  L^{q}(B_R), \qquad \text{ with } q< \frac{Nr(p-1)}{N-r}.$$
    In particular if  $f \in L^{N, \infty}(B_R)$, we have that $$\nabla u_p \in L^{ \infty}(B_R).$$
\end{cor}

Nevertheless, the higher regularity for the weak solutions to the equation in \eqref{1} can be obtained outside the degeneracy set of the equation (see \cite{Amb, AGPass, I13, Grimaldi} ).\\
The explicit expression of $\abs{\nabla u_p}$ given by Theorem \ref{Teo1} allows us to investigate the existence and the regularity of the second derivatives of the solution, outside the degeneracy set of the problem, through direct calculations. \\
Actually, here we are able to establish a higher differentiability result under a Lorentz  assumption on the datum $f$ (compare with \cite{Amb, AGPass, I12, Bra}). Indeed, we have the following
\begin{thm}\label{TeoD2}
Let $f \in L^{r, \infty}(B_R)$, with $1<r\leq N$, be radially decreasing.
Then there exists $u_p$ solution to \eqref{1} such that $$\nabla^2 u_p \in L^q_{loc}, \quad \forall  \, q < \frac{Nr(p-1)}{N+ r (p-2)}.$$
Specifically, we note that if $$f \in L^{N, \infty} (B_R)$$ then $$\nabla^2 u_p \in L^q_{loc}, \quad \forall  \, q < N.$$
\end{thm}
Note that previous results \cite{AGPass, Bra, Clop} , which establish the differentiability of a suitable function of the gradient that vanishes in the degeneracy set, do not imply the $L^q$ regularity of the second derivatives of $u_p$ obtained here.\\ 
Moreover, we are able to recover such results under weaker assumption on the datum $f$. More precisely, we have
\begin{thm}\label{TeoHmezzi}
Let $f $ be radially decreasing in $L^{r, \infty}(B_R)$ with
 $$  r>\frac{Np}{N(p-1)+ (2-p)}$$ then
 if $u_p$ is a solution to \eqref{1} for $p\geq2$, then
 $$( \abs{\nabla u_p} -1 )^{\frac{p}{2}}_{+} \, \frac{\nabla u_p}{\abs{\nabla u_p}} \in W^{1,2}(B_R).$$ 
\end{thm}
 Actually, by means of an example adapted from \cite{BraEsempio},  we show that Theorems \ref{TeoD2} and \ref{TeoHmezzi} and Corollary \ref{Cor1} are sharp in the Lorentz spaces setting.\\
$\quad$ \\
Moreover by Theorem \ref{Teo1} and after selecting $u_p$ as solution in $B_R$, we can analyze the asymptotic behaviour of $u_p$ as $p \to 1^+.$ Actually, arguing as in \cite{Merc}, we are able  to establish the following
\begin{thm}\label{TeoMerc}
Let $f \in L^{N, \infty}(B_R)$ with $\abs{\abs{f}}_{L^{N, \infty}(B_R)} \leq NC_{N}^{\frac{1}{N}}.$ Then there exists $u_p$ that solves problem  \eqref{1}, for any $p > 1$, and that converges a.e. in $B_R$ to a function $u \in W_0 ^ {1,1}$.\\
Moreover there exists a vector field $z$ such that
\begin{align*}
  &  z \in L^{\infty} (B_R, \mathbb{R}^N) \quad \text{ with } \quad\abs{\abs{z}}_{\infty} \leq 1;\\
   & - \mathrm{div} z =f; \\
  &  z \cdot \nu \leq 0 \qquad \mathscr{H}^{N-1}- \text{ a.e. on } \quad \partial B_R ;
\end{align*}
where $\nu$ denotes the outer normal to $\partial B_R$;
$$z \cdot \nabla u = \abs{\nabla u} \quad \text{as measures in } B_R.$$
\end{thm}
\noindent The main difference with respect to the arguments used in \cite{Merc} is that we have the explicit expression of the solution only outside the degeneracy set. Hence, we need to appropriately define the solution inside this set and then pass to the limit.\\

 \noindent Below we provide an overview of the paper's contents. After a breef introduction of the notations and definitions we use in the paper ( Section  \ref{sec:prelim}), we explore ( Section  \ref{sec:Approx}) a family of approximating problems and we present the proof of Theorem \ref{Teo1} and Corollary \ref{Cor1}. After  establishing some regularity for the second-order derivatives $\nabla^2 u_p$ of the weak solution $u_p$ of a family \eqref{1} (Section  \ref{sec:Regolarity}), we analyze the behaviour of the family $(\nabla u_p)_p$ as $p \to 1$ (Section \ref{sec:Stability}). Finally, we conclude with an example (Section   \ref{sec:Example}).

\begin{singlespace}

\section{Notations and preliminaries}\label{sec:prelim}
\end{singlespace}
In this section, we shall recall some tools and we fix notations and definitions that will be useful to prove our results.\\
We denote by $B_R \subset \mathbb{R}^N$ the open ball with radius $R>0$ centered at the origin, i.e. $$B_R = \, \{x \in  \mathbb{R}^{N}  :   \abs{x} < R \}.$$

Let $f: \Omega \subset \mathbb{R}^N \to \mathbb{R}$ be a real-valued measurable function..  
Its distribution function $\mu_f$  is defined by 
\begin{equation*}
    \mu_f (\tau) \, = \, \mu \left( \{ x \in \Omega \, \big| \, \abs{f(x)}> \tau \}\right), \hspace{4em} (\tau \geq 0),
\end{equation*}
where here and in the sequel, $\mu (A)$ denotes the Lebesgue measure of a measurable subset $A$ of $\mathbb{R}^{N}$.\\
The decreasing rearrangement of $f$ is the function $f^*$ defined on $[0, \infty)$ by 
\begin{equation*}
f^*(\lambda)= \inf \{\tau \geq 0 \,: \,   \mu_f (\tau)\leq \lambda  \}=\sup \{\tau \geq 0 \,: \,   \mu_f (\tau)> \lambda  \} , \hspace{3em} (\lambda \geq 0),
\end{equation*}
since $\mu_f$ is right-continuous and decreasing.
Observe that $f^*$ depends only on the absolute value $\abs{f}$.\\
\medskip 
\noindent Furthermore with the simbol $f^{**}$, we will denote the maximal function of $f^*$ defined as follows
\begin{equation}\label{f**}
f^{**}(t)= \frac{1}{t} \int_0 ^t f^*(s) ds, \hspace{4em}(t>0).
\end{equation}
The Lorentz space $L^{r, \infty}$ , for $1<r<\infty$, consists of all Lebesgue measurable functions $f$ such that
\begin{equation}\label{NormaLNINf}
    \abs{\abs{f}}_{L^{r,\infty}} = \sup_{0<t<\infty}{t^{\frac{1}{r}} f^{**} (t)} < + \infty.
\end{equation}
For further needs, we recall that 
\begin{equation}\label{fstar<fstarstar}
    f^* (t) \leq f^{**}(t) \quad \text{a.e.   } \, t>0.
\end{equation}
\noindent When dealing with widely degenerate \textit{p}-Laplace equations, the ellipticity bounds are expressed using the auxiliary function $H_{\alpha}(\xi) : \mathbb{R}^{N} \to  \mathbb{R}^{N}$ defined by  
\begin{equation}\label{H_alpha}
 H_{\alpha}(\xi) :=
\begin{cases}
    (\abs{\xi} -1 )^{\alpha} _+ \displaystyle{ \frac{\xi}{\abs{\xi}} } \hspace{4em} \text{  if  } \xi \in \mathbb{R}^{N} \setminus \{0\},\\
    0 \hspace{9em} \text{  if  } \xi =0,
    \end{cases}
\end{equation}
where $\alpha >0$ is a parameter.\\
 Indeed, it holds that
\begin{lem}\label{lemmaHalpha2}
    If $ \, 1 < p < \infty$, there exists a constant $\beta \equiv \beta(p,N)>0 $ such that
    \begin{equation*}
        \langle H_{p-1}(\xi)- H_{p-1}(\varsigma), \xi-\varsigma \rangle \, \geq \, \beta \abs{H_{\frac{p}{2}}(\xi) - H_{\frac{p}{2}}(\varsigma)}^2
    \end{equation*}
    and
      \begin{equation*}
          \abs{ H_{p-1}(\xi)- H_{p-1}(\varsigma)} \leq (p-1) \left( \abs{H_{\frac{p}{2}}(\xi)}^{\frac{p-2}{p}} + \abs{H_{\frac{p}{2}}(\varsigma)}^{\frac{p-2}{p}} \right) \abs{H_{\frac{p}{2}}(\xi) - H_{\frac{p}{2}}(\varsigma)}. 
    \end{equation*}
    for every $\xi, \varsigma \in \mathbb{R}^{N}$. In case $p \geq 2, \, \beta= \beta(p).$
\end{lem}
\noindent For the proof we refer to \cite[Lemma 4.1]{Bra} and \cite[Lemma 2.5]{Amb}.\\
Next lemma can be deduced by \cite[Lemma 8.3]{Giusti}.
\begin{lem}\label{LemmaCostante}
    For every $\alpha >0$ there exists a positive constant $c(\alpha)$ such that 
    \begin{equation*}
        \frac{1}{c} \Big\lvert  \abs{z}^{\alpha - 1}z - \abs{w}^{\alpha - 1}w \Big\rvert  \leq \abs{z-w}(\abs{z}+ \abs{w})^{\alpha - 1} \leq  c \Big\lvert  \abs{z}^{\alpha - 1}z - \abs{w}^{\alpha - 1}w \Big\rvert
    \end{equation*}
    for every $z, w \in \mathbb{R}^N$.
\end{lem}
For further needs, we record that 
\begin{align}\label{costantiNeed}
\Big\lvert  H_1 (\xi) - H_1(\eta) \Big\rvert &= \Big\lvert (\abs{\xi}-1)_+\frac{\xi}{\abs{\xi}} - (\abs{\eta}-1)_+\frac{\eta}{\abs{\eta}} \Big\rvert \notag \\
& \leq c \Big\lvert (\abs{\xi}-1)_+^{\frac{p}{2}-1}(\abs{\xi}-1)_+\frac{\xi}{\abs{\xi}} - (\abs{\eta}-1)_+^{\frac{p}{2}-1}(\abs{\eta}-1)_+\frac{\eta}{\abs{\eta}} \Big\rvert \, \cdot \notag \\
& \hspace{15mm} \cdot \Big( \,\abs{\xi}-1)_+ + (\abs{\eta}-1)_+ \, \Big)^{1 - \frac{p}{2}} 
\notag \\
& = \Big\lvert  H_{\frac{p}{2}} (\xi) - H_{\frac{p}{2}}(\eta) \Big\rvert \,\Big(\abs{\xi}-1)_+ + (\abs{\eta}-1)_+ \Big)^{1 - \frac{p}{2}} .
\end{align}
\\
\noindent Indeed it suffices to use Lemma \ref{LemmaCostante} with $z = (\abs{\xi}-1)_+\frac{\xi}{\abs{\xi}}$, $w= (\abs{\eta}-1)_+\frac{\eta}{\abs{\eta}}$ and $\alpha= \frac{p}{2}$.\\
\\
Finally, we recall the definition of weak solution.
\begin{defn}
    A function $u \in W_0^{1,p}(B_R) $ is a weak solution to \eqref{1} if and only if the following integral identity:
    \begin{equation*}
        \int_{B_R} \langle H_{p-1} (\nabla u), \nabla \varphi \rangle \, dz \,= \, \int_{B_R} f \varphi \, dz.
    \end{equation*}
    holds for any test function $\varphi \in C_0 ^\infty (B_R)$.\\
\end{defn}

Now, we retrieve a significant theorem, which is  a particular case  of  \cite[Theorem 1]{Tal}  suitable for our purposes. 
\begin{thm}\label{Talentii}
    For an open and bounded set $\Omega \subset \mathbb{R}^N,$ 
    let us consider 
    \begin{equation}\label{DriTALENTI}
    \begin{cases}
   - \mathrm{div}\Big( a \left(  \nabla v_p \right)  \Big) = \, f
   \qquad                    \, \mathrm{in} \quad   \Omega \qquad\\
       v_{p}=g \, \hspace{9em} \mathrm{on} \quad \partial{\Omega}.
      \end{cases}
      \end{equation}
    The hypotheses we assume are the following:
    \begin{enumerate}
        \item Ellipticity Condition: 
        there exists a function $A: [0, +\infty] \to [0, +\infty]$ such that
        \begin{equation*}
\begin{aligned}
 \text{i.  }& \, A (r) \text{ is convex in } \quad 0 \leq r \leq \infty, \\ 
 \\
\text{ii.  }& \, \lim_{r \to 0 } \frac{A(r)}{r} \, = \,0 \quad \text{ and } \, \lim_{r \to \infty} \frac{A(r)}{r} = \infty \\
\\
  \text{iii.  }& \, \langle a (\xi), \xi \rangle \, \geq A(\abs{\xi}) \quad \text{for all } \xi \in \mathbb{R}^N.
\end{aligned}
\end{equation*}
\item The right-hand side $f$ is measurable, integrable in 
 $\Omega$ and belong to some suitable Lorentz space. 
\item Suppose that $ g \in L^{\infty}(B_{R}) \cap W^{1,A} (B_R)$. 
\end{enumerate}
Then for a weak solution $v$ to  \eqref{DriTALENTI} from the convex Orlicz-Sobolev class  $W^{1,A} (B_R)$, i.e. such that $$\int_{B_R} A(\abs{\nabla u}) dx < \infty,$$
it holds the following 
\begin{equation*}\label{vminoreu}
 v^{*}(C_N \abs{x}^N) \leq  u(x),
 \end{equation*}
 where $u(x)$ is a weak solution of the symmetrized problem
  \begin{equation*}
    \begin{cases}
   - \mathrm{div}\Big( a \left(  \nabla u_p \right)  \Big) = \, f^*
   \qquad                    \, \mathrm{in} \quad   B_R \qquad\\
       u_{p}=g \, \hspace{9em} \mathrm{on} \quad \partial{B_R}.
      \end{cases}
      \end{equation*}
and it has the following explicit expression  
\begin{equation*}
 u(x)= \, sup \abs{g} \, + \, \int_{C_N \abs{x}^N}^{C_N R^N} B^{-1} \left(  \frac{r^{-1 + \frac{1}{N}}}{NC_{N}^{\frac{1}{N}}} \, \int_{0}^{r} f^*(s) ds \,\right)   \frac{r^{-1 + \frac{1}{N}}}{NC_{N}^{\frac{1}{N}}} dr,
\end{equation*}
where $\displaystyle {B(r)=\frac{A(r)}{r}}$ and $C_NR^N$ is the measure of the ball centered at the origin and radius $R$ in $\mathbb{R}^N$.
\end{thm}
\begin{singlespace}

\section{Proof of Theorem \ref{Teo1}}\label{sec:Approx}
\end{singlespace}
In order to prove Theorem \ref{Teo1}, we argue by approximation, introducing a family of uniformly elliptic problems. With this aim, for $\varepsilon \in (0,1]$, we consider the following family of Dirichlet symmetrized problems
\begin{equation}\label{2}
\begin{cases}
    - \mathrm{div} \left( \Bigg( \,\Big( \abs{\nabla u_{p}^{\varepsilon}} -1 \Big)_{+} +\varepsilon \abs{\nabla u_{p}^{\varepsilon}} \,\Bigg)^{p-1} \, \displaystyle{ \frac{\nabla u_{p}^{\varepsilon}}{\abs{\nabla u_{p}^{\varepsilon}}}} \right)= \, f   \qquad \text{                     in   }\, B_R \\
    u_{p}^{\varepsilon}=0 \hspace{21em} \text{                                on  } \partial{B_R},
\end{cases}
\end{equation}
where $f$ radially decreasing is the one fixed in problem \eqref{1}.\\
\textbf{Case $\mathbf{p \geq 2}$}:
We observe that, setting
\begin{equation*}
     a_{\varepsilon, p} (x, \xi) \, =\, \Big(\, (\abs{\xi} -1)_{+} +\varepsilon \abs{\xi} \,\Big)^{p-1} \frac{\xi}{\abs{\xi}},
\end{equation*}
the equation in \eqref{2} satisfy the assumptions of Theorem \ref{Talentii}, with 
\begin{equation*}
     A_{\varepsilon,p} (r) \, =\, \Big(\, (r -1)_{+} +\varepsilon r \,\Big)^{p-1}r,
\end{equation*}
with $r>0$.\\
More precisely, we can easily check that
\begin{align*}
 \text{i.  }& \, A_{\varepsilon,p} (r) \text{ is convex for } r \geq 0 \\ 
 \\
\text{ii.  }& \, \lim_{r \to 0 } \frac{A_{\varepsilon,p} (r)}{r} \, = \,0 \\
\\
  \text{iii.  }& \, \langle a_{\varepsilon,p} (\xi), \xi \rangle \, =\,A_{\varepsilon,p} (\abs{\xi}).
\end{align*}
Therefore we are legitimate to use Theorem \ref{Talentii} to deduce that the weak solution $u^\varepsilon _p \in W^{1, A_{\varepsilon}}(B_R)$ of \eqref{2} has the following expression for $p \geq2$: 
\begin{equation*} \label{SolTal}
        u^{\varepsilon} _{p} (x) = \int^{C_N R^N}_{C_N \abs{x}^N} B^{-1}_{\varepsilon,p} \left( \frac{s^{-1+ \frac{1}{N}}}{N C^{\frac{1}{N}}_{N}} \int^s _0  f^* (\sigma)\, d\sigma  \right)\frac{s^{-1+ \frac{1}{N}}}{N C^{\frac{1}{N}}_{N}} \, ds,
\end{equation*}
where $B_{\varepsilon,p} (r)=\frac{A_{\varepsilon,p} (r)}{r}$, and so
\begin{align}\label{Dupepsilon}
\nabla u^{\varepsilon}_p (x) &= -B^{-1}_{\varepsilon, p} \left( \frac{\abs{x}}{N C_{N} \abs{x}^N} \int^{C_N \abs{x}^N} _0  f^* (\sigma)\, d\sigma  \right) \frac{x}{\abs{x}} = \\ \notag
\\ \notag
& =-B^{-1}_{\varepsilon, p} \left( \frac{\abs{x}}{N } f^{**} (C_{N} \abs{x}^N)  \right) \frac{x}{\abs{x}},
\end{align}
where in the last equality we used the definition of $f^{**}(t)$ given at \eqref{f**}.\\
Note that with our choice of $A_{\varepsilon,p}(r)$, the Orlicz space $W^{1, A_{\varepsilon}}(B_R)$ coincides with $W^{1, p}(B_R)$. In fact if $ \displaystyle{\nabla h \in L^p (B_R)}$, we have
\begin{equation*}
   \int_{B_R} A_{\varepsilon, p}(\nabla h) dx = \int_{B_R} \Bigg( \,\Big( \abs{\nabla h} -1 \Big)_{+} +\varepsilon \abs{\nabla h} \,\Bigg)^{p-1} \abs{\nabla h} dx \leq (1+ \varepsilon)^{p-1} \int_{B_R} \abs{\nabla h }^p dx < +\infty.
\end{equation*}
On the other hand if $\int_{B_R} A_{\varepsilon,p}(\abs{\nabla h}) < +\infty$, we have
\begin{align*}
   \int_{B_R} \abs{\nabla h}^p dx &= \int_{B_R} \abs{\nabla h}^{p-1} \abs{\nabla h} dx\\
   &= \int_{B_R} \left( \abs{\nabla h} +1-1 \right)^{p-1} \abs{\nabla h} dx\\
   & \leq \int_{ B_R \cap \{\abs{ \nabla h} \leq 1 \}} dx + \int_{ B_R \cap \{\abs{ \nabla h} \geq 1 \}}  \Big( ( \abs{\nabla h} -1 )_{+} + \frac{1}{\varepsilon}\, \varepsilon \, \abs{\nabla h} \Big)^{p-1} \abs{\nabla h} dx\\
   &\leq \abs{B_R}+ \frac{1}{\varepsilon^{p-1}} \int_{B_R} A_{\varepsilon,p}(\abs{\nabla h}) dx < +\infty.
   \end{align*}
For further needs we record the following
\begin{equation}\label{ B^-1}
    B^{-1}_{\varepsilon,p}(s) = \begin{cases}
        \frac{1}{\varepsilon} s^{\frac{1}{p-1}}  \hspace{4em} 0 \leq s \leq \varepsilon^{p-1} \\
        \\
        \frac{1+s^{\frac{1}{p-1}}}{1+\varepsilon} \hspace{5em}  \,s > \varepsilon^{p-1} .
    \end{cases}
\end{equation}
\textbf{Case $\mathbf{1<p<2}$}: The weak solution of the problem \eqref{2} for $p \geq 2$, i.e.
\begin{equation*}\label{soltotale}
        u^{\varepsilon} _{p} (x) = \int^{C_N R^N}_{C_N \abs{x}^N} B^{-1}_{\varepsilon,p} \left( \frac{s^{-1+ \frac{1}{N}}}{N C^{\frac{1}{N}}_{N}} \int^s _0  f^* (\sigma)\, d\sigma  \right)\frac{s^{-1+ \frac{1}{N}}}{N C^{\frac{1}{N}}_{N}} \, ds,
\end{equation*}
proves to be the weak solution to \eqref{2} also for $1<p<2$. 
Indeed, by \eqref{Dupepsilon} we have
\begin{align*}
\abs{\nabla u^{\varepsilon}_p (x)} &= B^{-1}_{\varepsilon, p} \left( \frac{\abs{x}}{N } f^{**} (C_{N} \abs{x}^N)  \right)  = \\
&=\begin{cases}
  \frac{1}{\varepsilon^p} \left( \frac{\abs{x}}{N } f^{**} (C_{N} \abs{x}^N)  \right)^{\frac{1}{p-1}} \hspace{8em} \, \text{if   } \quad \displaystyle\frac{\abs{x}}{N } f^{**} (C_{N} \abs{x}^N) \leq \varepsilon^{p-1} \\
         \displaystyle \frac{1+ \left( \frac{\abs{x}}{N } f^{**} (C_{N} \abs{x}^N)  \right)^{ \frac{1}{p-1}}}{1+ \varepsilon}  
        \hspace{7em} \text{if   } \, \quad \frac{\abs{x}}{N } f^{**} (C_{N} \abs{x}^N) > \varepsilon^{p-1}
\end{cases}
\end{align*}
So, setting $J= \frac{\abs{x}}{N } f^{**} (C_{N} \abs{x}^N)$, we get
\begin{align*}
    \Bigg( \,\Big( \abs{\nabla u_{p}^{\varepsilon}} -1 \Big)_{+} +\varepsilon \abs{\nabla u_{p}^{\varepsilon}} \,\Bigg)^{p-1}= & \begin{cases}
     \left[  \displaystyle\varepsilon^p \frac{1}{\varepsilon^p} \left( \frac{\abs{x}}{N } f^{**} (C_{N} \abs{x}^N)  \right)^{\frac{1}{p-1}} \right]^{p-1} \hspace{2em} \, \text{if   } J \leq \varepsilon^{p-1} \\
           \\
         \displaystyle \left[ (1+ \varepsilon)\abs{\nabla u_{p}^{\varepsilon}} -1 \right]^{ p-1}  
        \hspace{7em} \text{if   }  J > \varepsilon^{p-1}
    \end{cases}\\
\\
=&
 \begin{cases}
       \displaystyle \frac{\abs{x}}{N } f^{**} (C_{N} \abs{x}^N)\quad  \hspace{14em} \, \text{if   } \quad  \displaystyle J \leq \varepsilon^{p-1} \\
           \\
         \displaystyle \left[ (1+ \varepsilon)\displaystyle \frac{1+ \left( \frac{\abs{x}}{N } f^{**} (C_{N} \abs{x}^N)  \right)^{ \frac{1}{p-1}}}{1+ \varepsilon} -1 \right]^{ p-1}  
        \hspace{3em}\,  \text{if   } \quad J > \varepsilon^{p-1}
    \end{cases} \\
    \\
    =&   \displaystyle \frac{\abs{x}}{N } f^{**} (C_{N} \abs{x}^N).
\end{align*}
Then we have
    \begin{align*}
        \Bigg( \,\Big( \abs{\nabla u_{p}^{\varepsilon}} -1 \Big)_{+} +\varepsilon \abs{\nabla u_{p}^{\varepsilon}} \,\Bigg)^{p-1} \, \displaystyle{ \frac{\nabla u_{p}^{\varepsilon}}{\abs{\nabla u_{p}^{\varepsilon}}}} \,= \,- \displaystyle \frac{\abs{x}}{N } f^{**} (C_{N} \abs{x}^N) \frac{x}{\abs{x}}\,=\,-\displaystyle \frac{x}{N } f^{**} (C_{N} \abs{x}^N) .
    \end{align*}
At this point we can calculate
\begin{align*}
         \mathrm{div} &\left( \Bigg( \,\Big( \abs{\nabla u_{p}^{\varepsilon}} -1 \Big)_{+} +\varepsilon \abs{\nabla u_{p}^{\varepsilon}} \,\Bigg)^{p-1} \, \displaystyle{ \frac{\nabla u_{p}^{\varepsilon}}{\abs{\nabla u_{p}^{\varepsilon}}}} \right)\,\\
         \qquad &= \,\mathrm{div} \left( - \displaystyle \frac{x}{N } f^{**} (C_{N} \abs{x}^N)\right)\\
         &= \sum_{i=1}^N \partial_{x_i} \left( - \displaystyle \frac{x_i}{C_N N \abs{x}^N } \int_0^{C_{N} \abs{x}^N}f^{*} (\sigma) d\sigma\right)\\
         &= \sum_{i=1}^N \Bigg( - \displaystyle \frac{1}{C_N N \abs{x}^N } \int_0^{C_{N} \abs{x}^N}f^{*} (\sigma) d\sigma  -x_i \partial_{x_i} \Big(  \displaystyle \frac{1}{C_N N \abs{x}^N } \int_0^{C_{N} \abs{x}^N}f^{*} (\sigma) d\sigma   \Big) \Bigg) \\
         &= - \displaystyle \frac{1}{C_N  \abs{x}^N } \int_0^{C_{N} \abs{x}^N}f^{*} (\sigma) d\sigma - \sum_{i=1}^N x_i \Bigg( \frac{1}{C_N N} \Bigg[-N \abs{x}^{-N-1}\frac{x_i}{\abs{x}} \int_0^{C_{N} \abs{x}^N}f^{*} (\sigma) d\sigma + \\
         &\qquad + \qquad \frac{1}{\abs{x}^N} f^{*}(C_N \abs{x}^N)C_N N \abs{x}^{N-1} \frac{x_i}{\abs{x}} \Big) \\
        &= - \frac{1}{C_N  \abs{x}^N } \int_0^{C_{N} \abs{x}^N} f^{*} (\sigma) \, d\sigma - \sum_{i=1}^N  \left( - \frac{1}{C_N} \frac{x_i^2}{\abs{x}^{N+2}} \int_0^{C_{N} \abs{x}^N} f^{*} (\sigma) \, d\sigma \right.\\
&\qquad \left. + \frac{x_i^2}{\abs{x}^{2}} f^{*}(C_N \abs{x}^N) \right)\\
&= - \frac{1}{C_N  \abs{x}^N } \int_0^{C_{N} \abs{x}^N} f^{*} (\sigma) \, d\sigma -  \left( - \frac{1}{C_N} \frac{1}{\abs{x}^{N}} \int_0^{C_{N} \abs{x}^N} f^{*} (\sigma) \, d\sigma + f^{*}(C_N \abs{x}^N) \right)\\
&= - f^{*}(C_N \abs{x}^N),
    \end{align*}
Therefore, we have shown
\begin{equation*}
      - \mathrm{div} \left( \Bigg( \,\Big( \abs{\nabla u_{p}^{\varepsilon}} -1 \Big)_{+} +\varepsilon \abs{\nabla u_{p}^{\varepsilon}} \,\Bigg)^{p-1} \, \displaystyle{ \frac{\nabla u_{p}^{\varepsilon}}{\abs{\nabla u_{p}^{\varepsilon}}}} \right)= \, f .
\end{equation*}
As a consequence the explicit weak solution we obtained using Talenti's Theorem for 
$p \geq 2$ also serves as a weak solution for $1<p<2$. In other words, we utilized Talenti's Theorem to determine the solution for $p \geq 2$, and subsequently demonstrated that this solution satisfies our equation for every $p>1$. 
\begin{proof}[Proof of Theorem \ref{Teo1} :]\label{proof1}
Using the expression of $\nabla u^{\varepsilon}_p (x)$ at \eqref{Dupepsilon} and the expression of $B^{-1}_{\varepsilon, p}(s)$ at \eqref{ B^-1} together with the assumpton $f \in L^{N, \infty}(B_R)$, we deduce that
\begin{align}
   \int_{B_R} \abs{\nabla u^{\varepsilon}_p (x)}^{p} dx &= \int_{B_R} \left(B^{-1}_{\varepsilon,p} \left( \frac{\abs{x}}{N } f^{**} (C_{N} \abs{x}^N)  \right) \right)^{p} dx \notag\\
  \notag \\
   & \leq \int_{B_R} \left(B^{-1}_{\varepsilon, p} \left( \frac{1}{NC^{\frac{1}{N}} _N } \abs{\abs{f}}_{L^{N, \infty}}  \right) \right)^{p} dx \notag \\
   \notag\\
   &=  \abs{B_R} \begin{cases}
           \frac{1}{\varepsilon^p} \left( \frac{1}{NC^{\frac{1}{N}} _N } \abs{\abs{f}}_{L^{N, \infty}} \right)^{\frac{p}{p-1}} \hspace{8em} \, \text{if   } \quad  \frac{\abs{\abs{f}}_{L^{N, \infty}}}{NC^{\frac{1}{N}}_N} \leq \varepsilon^{p-1} \\
           \\
         \displaystyle\Bigg( \frac{1+ \Big( \frac{1}{NC^{\frac{1}{N}} _N } \abs{\abs{f}}_{L^{N, \infty}} \Big)^{ \frac{1}{p-1}}}{1+ \varepsilon} \Bigg)^{p} 
        \hspace{6em} \text{if   } \quad \frac{\abs{\abs{f}}_{L^{N, \infty}}}{NC^{\frac{1}{N}}_N} > \varepsilon^{p-1} , 
           \end{cases} \label{IntBR}
   \end{align}
where in the second line of the previous estimate we used that $B^{-1} _{\varepsilon, p}(s)$ is increasing for $s \geq 0$ and the definition of the norm $\abs{\abs{f}}_{L^{N, \infty}}$ given at \eqref{NormaLNINf}.\\
   Estimate \eqref{IntBR} implies that
   \begin{equation} \label{IntDup}
       \int_{B_R} \abs{\nabla u^{\varepsilon}_p (x)}^{p} dx \leq \, 2^p C_N  R^N  \, \left[  1+\left( \frac{\abs{\abs{f}}_{L^{N, \infty}}}{NC^{\frac{1}{N}} _N } \right)^{\frac{p}{p-1}}    \right],
   \end{equation}
   and so the sequence $(\nabla u^{\varepsilon} _p)_\varepsilon$ has norm bounded in $L^p$ independently of $\varepsilon$, for each fixed $p>1$.\\
    Our next aim is to prove that the sequence $H_{\frac{p}{2}} (\nabla u^{\varepsilon}_p)$ strongly converges to $H_{\frac{p}{2}} (\nabla u_p$) in $L^2$.
Since $u^{\varepsilon}_p$ solves \eqref{2} and $u_p$ solves \eqref{1} we have
\begin{align}
\int_{B_R} \Big\langle \,( \abs{\nabla u_{p}} -1)_{+}^{p-1} \, \frac{\nabla u_{p}}{\abs{\nabla u_{p}}} \, , \nabla \varphi \Big\rangle \, &=
\int_{B_R} f \, \varphi = \notag \\
&= \int_{B_R}  \Big\langle \,( \abs{\nabla u^{\varepsilon}_p} -1)_{+}\, + \varepsilon \abs{\nabla u^{\varepsilon}_p} \,)^{p-1} \, \frac{\nabla u^{\varepsilon}_p}{\abs{\nabla u^{\varepsilon}_p}} , \, \nabla \varphi \Big\rangle , \notag
\end{align}
$\forall \varphi \in C_0 ^\infty (B_R)$ and obviously, by density, also for every $ \varphi \in W^{1,p}_0 (B_R)$. \\
Therefore, we may choose as test function in the previous identity  $\varphi =u_p - u^{\varepsilon}_p \in W^{1,p}_0 (B_R)$, thus getting
\begin{equation*}
\int_{B_R}  \left\langle \,\left( \, \left( \abs{\nabla u^{\varepsilon}_p} -1 \right)_{+}\, + \varepsilon \abs{\nabla u^{\varepsilon}_p} \,\right)^{p-1} \, \frac{\nabla u^{\varepsilon}_p}{\abs{\nabla u^{\varepsilon}_p}} \, -  \,\left( \abs{\nabla u_{p}} -1 \right)_{+}^{p-1} \, \frac{\nabla u_{p}}{\abs{\nabla u_{p}}} , \, \left(\nabla u_p - \nabla u^{\varepsilon}_p \right)\right\rangle  dx = \, 0
\end{equation*}
or equivalently
 \begin{align} 
 \int_{B_R} \bigg\langle  \, \left( \,\left( \abs{\nabla u^{\varepsilon}_p} -1\right)_{+}\, + \varepsilon \abs{\nabla u^{\varepsilon}_p} \,\right)^{p-1} \, \frac{ \nabla u^{\varepsilon}_p}{\abs{\nabla u^{\varepsilon}_p}} & - \left( \, \abs{\nabla u^{\varepsilon}_p} -1 \right)_{+}^{p-1} \, \frac{\nabla u^{\varepsilon}_p}{\abs{\nabla u^{\varepsilon}_p}} ,\, \left(\nabla u_p - \nabla u^{\varepsilon}_p \right) \bigg\rangle \, dx \notag \\ 
\notag \\
+ \int_{B_R} \bigg\langle  \, \left( \, \abs{\nabla u^{\varepsilon}_p} -1 \right)_{+}^{p-1} \, \frac{\nabla u^{\varepsilon}_p}{\abs{\nabla u^{\varepsilon}_p}}  &- \left( \, \abs{\nabla u_{p}} -1 \right)_{+}^{p-1} \, \frac{\nabla u_{p}}{\abs{\nabla u_{p}}} \, , \left( \nabla u_p -  \nabla u^{\varepsilon}_p \right) \,\bigg\rangle dx= \, 0 ,\notag
\end{align}
and so
\begin{equation}\label{DUpe-DUp}
 \begin{aligned}
 \int_{B_R} &\bigg\langle  \, \left( \, \abs{\nabla u^{\varepsilon}_p} -1 \right)_{+}^{p-1} \, \frac{\nabla u^{\varepsilon}_p}{\abs{\nabla u^{\varepsilon}_p}}  - \left( \, \abs{\nabla u_{p}} -1 \right)_{+}^{p-1} \, \frac{\nabla u_{p}}{\abs{\nabla u_{p}}} \, , \left(    \nabla u^{\varepsilon}_p - \nabla u_p \right) \bigg\rangle dx \\
 \\
 & =\int_{B_R} \bigg\langle \bigg[ \, \left( \,\left( \abs{\nabla u^{\varepsilon}_p} -1\right)_{+}\, + \varepsilon \abs{\nabla u^{\varepsilon}_p} \,\right)^{p-1} \,  - \left( \, \abs{\nabla u^{\varepsilon}_p} -1 \right)_{+}^{p-1} \, \,\bigg]\frac{\nabla u^{\varepsilon}_p}{\abs{\nabla u^{\varepsilon}_p}} , \left(\nabla u_p - \nabla u^{\varepsilon}_p \right) \bigg\rangle dx.
\end{aligned}
\end{equation}
By the definition at \eqref{H_alpha}, the left hand side of previous identity can be written as follows
\begin{align}
 \int_{B_R}& \bigg\langle   \, \left( \, \abs{\nabla u^{\varepsilon}_p} -1 \right)_{+}^{p-1} \, \frac{\nabla u^{\varepsilon}_p}{\abs{\nabla u^{\varepsilon}_p}}  - \left( \, \abs{\nabla u_{p}} -1 \right)_{+}^{p-1} \, \frac{\nabla u_{p}}{\abs{\nabla u_{p}}} \, , \left(    \nabla u^{\varepsilon}_p - \nabla u_p \right) \bigg\rangle dx = \notag \\
 &=\int_{B_R} \langle \, H_{p-1}(\nabla u^{\varepsilon}_p)  \, - H_{p-1} (\nabla u_p ) \, , \, \nabla u^{\varepsilon}_p -  \nabla u_p \rangle dx . \label{DU=H}
 \end{align}
Next, by virtue of Lemma \ref{lemmaHalpha2}, we get
 \begin{equation}\label{HMAGG}
 \int_{B_R} \langle \, H_{p-1}(\nabla u^{\varepsilon}_p)  \, - H_{p-1} (\nabla u_p ) \, , \, \nabla u^{\varepsilon}_p -  \nabla u_p \rangle dx \geq c(p, N) \int_{B_R} \left\lvert  H_{\frac{p}{2}}(\nabla u^{\varepsilon}_p)  \, - H_{\frac{p}{2}} (\nabla u_p) \right\lvert^2  dx.
 \end{equation}
Hence, combining \eqref{DU=H} and \eqref{HMAGG}, we conclude that
\begin{equation}\label{IntHp}
 \begin{aligned}
 \int_{B_R} &\bigg\langle  \, \left( \, \abs{\nabla u^{\varepsilon}_p} -1 \right)_{+}^{p-1} \, \frac{\nabla u^{\varepsilon}_p}{\abs{\nabla u^{\varepsilon}_p}}  - \left( \, \abs{\nabla u_{p}} -1 \right)_{+}^{p-1} \, \frac{\nabla u_{p}}{\abs{\nabla u_{p}}} \, , \left(    \nabla u^{\varepsilon}_p - \nabla u_p \right) \bigg\rangle dx\\
 &\geq c(p, N) \int_{B_R} \left\lvert  H_{\frac{p}{2}}(\nabla u^{\varepsilon}_p)  \, - H_{\frac{p}{2}} (\nabla u_p) \right\lvert^2 dx.
 \end{aligned}
 \end{equation}
The right hand side of \eqref{DUpe-DUp} can be estimated as follows
\begin{align}\label{IntRight}
& \left\lvert \int_{B_R}  \bigg\langle \left( \left( \lvert \nabla u^{\varepsilon}_p \rvert - 1 \right)_+ + \varepsilon \lvert \nabla u^{\varepsilon}_p \rvert \right)^{p-1} - \left( \lvert \nabla u^{\varepsilon}_p \rvert - 1 \right)_+^{p-1}  \frac{\nabla u^{\varepsilon}_p}{\lvert \nabla u^{\varepsilon}_p \rvert}, \left( \nabla u_p - \nabla u^{\varepsilon}_p \right) \bigg\rangle dx\right\rvert  \notag \\
&  \leq \int_{B_R} \left\lvert  \, \left( \,\left( \abs{\nabla u^{\varepsilon}_p} -1\right)_{+}\, + \varepsilon \abs{\nabla u^{\varepsilon}_p} \,\right)^{p-1} \,  - \left( \, \abs{\nabla u^{\varepsilon}_p} -1 \right)_{+}^{p-1} \, \,\right\lvert   \abs {\nabla u_p - \nabla u^{\varepsilon}_p }  dx \notag \\
&\leq c_{p}\, \int_{B_R}   \Big|  \,( \abs{\nabla u^{\varepsilon}_p} -1)_{+}\, + \varepsilon \abs{\nabla u^{\varepsilon}_p} \,  \,- \left( \abs{\nabla u^{\varepsilon}_p} -1 \right)_{+} \Big| \cdot \, \notag\\
&  \hspace{6em}  \cdot \left( \,\left( \abs{\nabla u^{\varepsilon}_p} -1\right)_{+}^{2} + \left( \varepsilon \abs{\nabla u^{\varepsilon}_p} \,\right)^{2} \right)^{\frac{p-2}{2}} \, \abs{\nabla u_p -  \nabla u^{\varepsilon}_p} dx  \notag\\
& = c_{p} \varepsilon \,\int_{B_R}  \abs{\nabla u^{\varepsilon}_p}\,\abs{\nabla u_p -  \nabla u^{\varepsilon}_p}
    \left( \,\left( \abs{\nabla u^{\varepsilon}_p} -1\right)_{+}^{2} +\varepsilon^2 \abs{\nabla u^{\varepsilon}_p}^{2} \right)^{\frac{p-2}{2}} dx:= \text{I}.
\end{align}
where we used Lemma \ref{LemmaCostante} with $z=\left( \,  \left( \abs{\nabla u^{\varepsilon}_p} -1\right)_{+}\, + \varepsilon \abs{\nabla u^{\varepsilon}_p} \, \right)$, $w=\left( \, \abs{\nabla u^{\varepsilon}_p} -1 \right)_{+}$ and $\alpha = p-1.$\\
Since we are dealing with all exponents $p>1$, we are going to estimate last integral in \eqref{IntRight} separating the cases $1<p<2$ and $p \geq 2$.\\
Let us suppose first $p \geq 2$. In this case, since $p-2 \geq 0$, we have
\begin{align}
   \text{I} =  c_{p} \varepsilon \, &\int_{B_R}  \abs{\nabla u^{\varepsilon}_p}\,\abs{\nabla u_p -  \nabla u^{\varepsilon}_p}
    \left( \,\left( \abs{\nabla u^{\varepsilon}_p} -1\right)_{+}^{2} +\varepsilon^2 \abs{\nabla u^{\varepsilon}_p}^{2} \right)^{\frac{p-2}{2}} dx \notag\\
    \notag\\
    & \leq c_{p} \varepsilon\, \int_{B_R}  \abs{\nabla u^{\varepsilon}_p}\,\abs{\nabla u_p -  \nabla u^{\varepsilon}_p}
     \, \Big(  \abs{\nabla u^{\varepsilon}_p}^{2} +\varepsilon^2 \abs{\nabla u^{\varepsilon}_p}^{2}  \Big)^{\frac{p-2}{2}} dx \notag\\
    \notag \\
     & \leq c_{p}\, \varepsilon (1+ \varepsilon^2)^{\frac{p-2}{2}}\, \int_{B_R}      \abs{\nabla u^{\varepsilon}_p}^{p-1}\,\abs{\nabla u_p -  \nabla u^{\varepsilon}_p} \, dx \notag\\
    \notag \\
     & \leq  2^{\frac{p-2}{2}}c_{p}\, \varepsilon \, \left( \int_{B_R}      \abs{\nabla u^{\varepsilon}_p}^{p} dx + \int_{B_R} \abs{\nabla u^{\varepsilon}_p}^{p-1} \, \abs{\nabla u_p} \, dx \right) \notag\\
    \notag \\
     & \leq  c_{p} \varepsilon\, \left( \int_{B_R}      \abs{\nabla u^{\varepsilon}_p}^{p} dx + \int_{B_R}  \abs{\nabla u_p}^{p} dx \right), \notag \label{Pmagg1}
    \end{align}
where we used  Young's inequality and that $\varepsilon \leq 1$. \\
Now, suppose that $1<p<2$. In this case, since $p-2 <0$, we have that 
\begin{equation*}
    \Big( \, ( \abs{\nabla u^{\varepsilon}_p}-1)_{+} ^{2} \, + \, \varepsilon^2 \abs{\nabla u^{\varepsilon}_p}^{2} \, \Big)^{\frac{p-2}{2}} \leq \varepsilon^{p-2} \abs{\nabla u_{p}^{\varepsilon}}^{p-2}
\end{equation*}
and so
    \begin{align}
  \text{I} &=   c_{p}\,\varepsilon \, \int_{B_R}  \abs{\nabla u^{\varepsilon}_p}\,\abs{\nabla u_p -  \nabla u^{\varepsilon}_p}
    \left( \,\left( \abs{\nabla u^{\varepsilon}_p} -1\right)_{+}^{2} +\varepsilon^2 \abs{\nabla u^{\varepsilon}_p}^{2} \right)^{\frac{p-2}{2}} \, dx  \notag \\
   \notag \\
      & \leq c_{p}\, \varepsilon^{p-1}\, \int_{B_R}      \abs{\nabla u^{\varepsilon}_p}^{p-1}\,\abs{\nabla u_p -  \nabla u^{\varepsilon}_p} \, dx  \notag \\ 
      \notag  \\
     & \leq  c_{p}\,\varepsilon^{p-1}\, \left( \int_{B_R}      \abs{\nabla u^{\varepsilon}_p}^{p} \, dx  + \int_{B_R}  \abs{\nabla u_p}^{p} \, dx \right), \notag
    \end{align}
where we used Young's inequality again. 
Hence, joining both cases we conclude that
\begin{align}\label{bothPmaggcompreso}
   I \leq c_p (\varepsilon+\varepsilon^{p-1}) \, \left( \int_{B_R}      \abs{\nabla u_p}^{p} \,dx  + \, \int_{B_R}  \abs {\nabla u^{\varepsilon}_p}^p \, dx \right)
   \end{align}
At this point, in both cases, inserting \eqref{IntHp}, \eqref{IntRight}, \eqref{bothPmaggcompreso} in  \eqref{DUpe-DUp}, we have
\begin{align*}
      \int_{B_R} \left\lvert \, H_{\frac{p}{2}}(\nabla u^{\varepsilon}_p)  \, - H_{\frac{p}{2}} (\nabla u^p) \right\lvert ^2 dx &\leq c_{p}\,(\varepsilon+\varepsilon^{p-1}) \,\left[ \left( \int_{B_R}      \abs{\nabla u_p}^{p} dx \right)^{\frac{1}{p}}\, + \, \left( \int_{B_R}  \abs {\nabla u^{\varepsilon}_p}^p dx\right)^\frac{1}{p}\, \right] \\
      & \leq c_{p}\,(\varepsilon+\varepsilon^{p-1}) \,\left[ \left( \int_{B_R}      \abs{\nabla u_p}^{p} dx\right)^{\frac{1}{p}}\, + \, 2^p \,  C_N \, R^N \, \left(  1+ \frac{\abs{\abs{f}}_{L^{N, \infty}}}{NC^{\frac{1}{N}} _N }    \right)^{\frac{1}{p-1}}  \right],
\end{align*}
where last inequality is due to the estimate \eqref{IntDup}. \\
Taking the limit as $\varepsilon \rightarrow 0$  in previous inequality, we obtain:
\begin{equation*}
   \lim_{\varepsilon \to 0} \int_{B_R} \left\lvert \, H_{\frac{p}{2}}(\nabla u^{\varepsilon}_p)  \, - H_{\frac{p}{2}} (\nabla u_p) \right\lvert^2 dx \,= \,0,
\end{equation*}
i.e. $$ H_{\frac{p}{2}}(\nabla u^{\varepsilon}_p) \to  H_{\frac{p}{2}}(\nabla u_p) \text{ strongly in }L^2 (B_{R}).$$ 
Therefore, up to a not relabeled subsequence, we also have that
\begin{align}\label{Hconvergenza}
    H_{\frac{p}{2}}(\nabla u^{\varepsilon}_p )\to  H_{\frac{p}{2}}(\nabla u_p) \qquad \text{a.e. in }\, B_R.
\end{align}

Using \eqref{costantiNeed} and \eqref{Hconvergenza}, we deduce that
\begin{align}
( \abs{\nabla u^{\varepsilon}_p} -1)_{+} \, \frac{\nabla u^{\varepsilon}_p}{\abs{\nabla u^{\varepsilon}_p}} \, \rightarrow \, ( \abs{\nabla u_p} -1 )_{+} \, \frac{\nabla u_p}{\abs{\nabla u_p}}  \quad \mathrm{   a.e. \, in  }\, B_R. \notag
\end{align}
Recalling \eqref{Dupepsilon}
and observing that $$ \lim_{\varepsilon \to 0^+} B^{-1}_{\varepsilon,p}(s) =  1+ s ^{\frac{1}{p-1}} =: \Tilde{B}_p (s) \quad \text{ a.e. in }(0, +\infty), $$  
by the continuity of the function $(\abs{t}-1)_+ \frac{t}{\abs{t}}$, we get
\begin{align}\label{(Du-1)Du}
     ( \abs{\nabla u_p} -1 )_{+} \, \frac{\nabla u_p}{\abs{\nabla u_p}} = 
     -\left( \frac{ \abs{x}}{N}f^{**}(C_N \abs{x}^N)\right)^{\frac{1}{p-1}}\, \frac{x}{\abs{x}} \quad \mathrm{ a.e. \, in \,} \, B_R.
\end{align}
The map $\, t \rightarrow (\abs{t}-1)_+ \frac{t}{\abs{t}}$ is invertible in $\mathbb{R}^N \setminus \{ \abs{t} \leq 1 \}$ and its inverse is the map is given by $$s \rightarrow \frac{\abs{s}+1}{\abs{s}}s.$$ Therefore from \eqref{(Du-1)Du}, we deduce that

    \begin{align*}
        \nabla u_p =- \left[\,  1 \,+ \, \left( \frac{\abs{x}}{N}f^{**}(C_N \abs{x}^N)\right)^{\frac{1}{p-1}} \,\right]\, \frac{x}{\abs{x}},
    \end{align*}
a.e. in the set where $\{ \abs{\nabla u_p}>1   \},$ i.e. the conclusion.
\end{proof}
Now, we can immediately give the
\begin{proof}[Proof of Corollary \ref{Cor1} :]
By virtue of \eqref{NormaLNINf}, we have
\begin{equation}\label{f**cor}
f^{**}(C_N \abs{x}^N) \leq \lvert \lvert f \rvert \rvert_{L^{r, \infty}} (C_N \abs{x}^N)^{- \frac{1}{r}}
\end{equation}
    Then
    \begin{align}
       \int_{B_R} \abs{\nabla u_p (x)}^{q} dx &=
       \int_{B_R \cap \{ \abs{ \nabla u_p} \leq 1\}} \abs{\nabla u_p (x)}^{q} dx +  \int_{B_R \cap \{ \abs{ \nabla u_p} > 1\}} \abs{\nabla u_p (x)}^{q} dx \notag \\
       \notag \\
     & \leq \abs{B_R} + \int_{B_R \cap \{ \abs{ \nabla u_p} > 1\}} \left[\,  1 \,+ \, \left( \frac{\abs{x}}{N}f^{**}(C_N \abs{x}^N)\right)^{\frac{1}{p-1}} \,\right]^q dx \notag \\
     \notag \\
     &\leq \, C  R^N \, + C \int_{B_R \cap \{ \abs{\nabla u_p} > 1\}} \abs{x}^{\frac{q}{p-1}- (\frac{N}{r})\frac{q}{p-1}} \, dx. \notag
   \end{align}
   In order to have $\abs{\nabla u_p} \in L^q(B_R)$, it suffices to satisfy the following condition $$\frac{q}{p-1} \left(1- \frac{N}{r} \right) + N>0,$$  i.e. if $f \in L^{r, \infty}(B_R)$, then $\abs{\nabla u_p} \in L^q(B_R)$ with $q < \frac{Nr(p-1)}{N-r}$.\\
   Moreover if $r=N$,  from \eqref{f**cor} follows
    $$C_N \abs{x} f^{**}(C_N \abs{x}^N) \leq \abs{\abs{f}}_{L^{N, \infty}} < \infty, \qquad \forall x \in B_R,$$ and so
    \begin{equation*}
          \abs{\abs{\nabla u_p }}_{L^{\infty}(B_R)} \leq 1+ \, \Big\lvert \Big\lvert  1 \,+ \, \left( \frac{\abs{x}}{N}f^{**}(C_N \abs{x}^N)\right)^{\frac{1}{p-1}}\Big\rvert \Big\rvert_{L^{\infty}(B_R)} \,\leq c \left(1 +  \, \abs{\abs{f}}_{L^{N, \infty}}^{\frac{1}{p-1}} \right), 
    \end{equation*}
    and so $\nabla u_p \in L^{\infty}(B_R)$.
\end{proof}
\begin{singlespace}
\section{Second order regularity}\label{sec:Regolarity}
\end{singlespace}
Since widely degenerate problems lose the uniqueness of their solutions, we can have many different solutions. However, we can extend $u_p$, defined at \eqref{MIAsolu}, to the set $\{ \abs{\nabla u_p} \leq 1\}$, choosing, as a weak solution to \eqref{1} in the set $\{\abs{\nabla u_p } \leq 1 \}$, the following
      \begin{equation}\label{NuovaSo}
        \tilde{u}_p (x) = \frac{1}{N C^{\frac{1}{N}}_{N}}\, \displaystyle{\int\limits^{C_N R^N}_{C_N \abs{x}^N} } \left( 1 + \left( \frac{s^{ \frac{1}{N}}}{N C^{\frac{1}{N}}_{N}}   f^{**}(s)  \right)^{\frac{1}{p-1}} \right) s^{-1+ \frac{1}{N}} \, ds,
\end{equation}
   in fact, with a simple calculation, we obtain 
    \begin{equation}\label{Dutild}
        \nabla \tilde{u}_p =- \left[\,  1 \,+ \, \left( \frac{\abs{x}}{N}f^{**}(C_N \abs{x}^N)\right)^{\frac{1}{p-1}} \,\right]\, \frac{x}{\abs{x}},
    \end{equation}
    which is exactly \eqref{Dupexplicit} found in Theorem \ref{Teo1}, in the set $\{ \abs{\nabla u_p } > 1 \}$. Then we can choose $\tilde{u}_p=u_p$ as a solution of \eqref{1} a.e. in the whole $B_R \in \mathbb{R}^N$.\\
The aim of this section is to establish some regularity results for the second-order derivatives $\nabla^2 u_p$ of the weak solution $u_p$ of \eqref{1} defined at \eqref{NuovaSo}. 
\begin{proof}[Proof of Theorem \ref{TeoD2}:]
By virtue of equality \eqref{Dutild} we can directly calculate the second order derivatives of $u_p$
\begin{align}
    \nabla^2 u_p &=  - \Bigg[   \frac{1}{p-1} \, \Bigg( \frac{\abs{x}^{1-N}}{NC_N} \int ^{C_N \abs{x}^N} _0 f^{*}(\sigma ) d \sigma \Bigg)^{\frac{2-p}{p-1}} \Bigg(  \frac{(1-N) \abs{x}^{-N}}{NC_N} \frac{x}{\abs{x}} \int ^{C_N \abs{x}^N} _0 f^{*}(\sigma ) d \sigma \, + \notag \\
    \notag \\
    & \quad + \frac{\abs{x}^{1-N}}{NC_N}  f^*(C_N \abs {x}^N)(C_N N \abs{x}^{N-2} x ) \Bigg) \frac{x}{\abs{x}} + \Bigg(\,  1 \,+ \, \left( \frac{\abs{x}^{1-N}}{NC_N} \int ^{C_N \abs{x}^N} _0 f^{*}(\sigma ) d \sigma \right)^{\frac{1}{p-1}} \,\Bigg) \cdot \notag\\
    \notag\\
    & \quad \cdot \Bigg( \frac{I}{\abs{x}} - \frac{x \otimes x}{\abs{x}^3}\Bigg)\Bigg] \,  \notag\\
    \notag\\
    &= \, - \Bigg[   \frac{1}{p-1} \, \Bigg( \frac{\abs{x}}{N} f^{**}(C_N \abs{x}^N ) \Bigg)^{\frac{2-p}{p-1}} \Bigg(  \frac{(1-N)} {N}f^{**}(C_N \abs{x}^N )  \frac{x}{\abs{x}} + f^{*}(C_N \abs{x}^N ) \frac{x}{\abs{x}} \Bigg) \frac{x}{\abs{x}} + \notag \\
    \notag \\
    & \quad +    \Bigg(\,  1 \,+ \, \left(\frac{\abs{x}}{N} f^{**}(C_N \abs{x}^N ) \right)^{\frac{1}{p-1}} \,\Bigg) \Bigg( \frac{I}{\abs{x}} - \frac{x \otimes x}{\abs{x}^3}\Bigg)\Bigg], \notag \label{derivataseconda}
\end{align}
where, in the last line,  we used the definition of the function $f^{**}(t)$ given at \eqref{f**}. \\
Now, recalling the norm's definition $\abs{\abs{f}}_{L^{r, \infty}}$ given at \eqref{NormaLNINf} for $f \in L^{r, \infty}$ with $1<r \leq N$, setting
\begin{equation*}
    M= \abs{\abs{f}}_{L^{r,\infty}(B_R)},
\end{equation*}
by virtue of \eqref{fstar<fstarstar}, we have
\begin{align}\label{stimeperlederivate}
&f^{*}(C_N \abs{x}^N) \leq f^{**}(C_N \abs{x}^N) \leq M (C_N \abs{x}^N)^{- \frac{1}{r}}.
\end{align}
To determine the exponents $q$ such that  $\abs{ \nabla^2 u_p} \in L^q(B_R)$, we can calculate
\begin{align}
  \int _{B_R} \abs{\nabla ^2 u_p}^q dx &\leq \, c
     \Bigg[   \int_{B_R} \Bigg( \abs{x}f^{**}(C_N \abs{x}^N ) \Bigg)^{\frac{2-p}{p-1}q} \Bigg(  f^{**}(C_N \abs{x}^N )  \Bigg)^{q} dx + \notag \\
   \notag  \\
    & \quad +    \int_{B_R} \frac{1}{\abs{x}^q} dx \,+ \, \int_{B_R}\left(\abs{x} f^{**}(C_N \abs{x}^N ) \right)^{\frac{q}{p-1}}\frac{1}{\abs{x}^q} dx \Bigg]  \notag \\
    \notag \\
    & = c  \, \Bigg[   \int_{B_R} \Bigg( \abs{x} ^{\frac{2-p}{p-1}q}\Bigg) \Bigg(  f^{**}(C_N \abs{x}^N )   \Bigg)^{\frac{q}{p-1}} dx +     \int_{B_R} \frac{1}{\abs{x}^q} dx \Bigg] \notag \\
    \notag \\
     & \leq c  \, \Bigg[   \int_{B_R} \Bigg( \abs{x} ^{\frac{2-p}{p-1}q}\Bigg) \Bigg(  M C_N^{-\frac{1}{r}} \abs{x}^{-\frac{N}{r}}   \Bigg)^{\frac{q}{p-1}} dx +     \int_{B_R} \frac{1}{\abs{x}^q}dx \Bigg] \notag \\
     \notag \\
    & \leq c  \, \Bigg[   \int_{B_R} \abs{x}^{(\frac{2-p}{p-1})q - \frac{N}{r}\frac{q}{p-1}} dx  +    \int_{B_R} \frac{1}{\abs{x}^q} dx \Bigg]. \label{dersecondeLq} \notag
    \end{align}

Then the integrability of $\abs{\nabla^2 u_p}^q$ is guaranteed if the exponents $q$ and $r$ satisfy the following conditions
\begin{equation}\label{CondizioniDersec}
    \begin{cases}
      \left( \frac{2-p}{p-1} \right)q - \frac{N}{r}\frac{q}{p-1} +N >0\\
      -q +N > 0.
    \end{cases}
\end{equation}
Since
\begin{equation*}
    \left( \frac{2-p}{p-1} \right) q - \frac{N}{r}\frac{q}{p-1} < -q
\end{equation*}
it suffices that $q$ and $r$ satisfy the first inequality in \eqref{CondizioniDersec}.
This is true provided 
\begin{equation}\label{rmagg}
    r> \frac{Nq}{N(p-1)+ q(2-p)} >1
\end{equation}
or equivalently that
\begin{equation*}
    q< \frac{Nr(p-1)}{N+ r (p-2)}.
\end{equation*}\end{proof}
\noindent In order to compare the properties of the solution to \eqref{1} with previous results involving higher differentiability, we now illustrate the conditions under which the function $ H_{\frac{p}{2}}(\abs{\nabla u_p}) \in W^{1,2}(B_R).$ 
\begin{proof}[Proof of Theorem \ref{TeoHmezzi} :]
    We first observe
\begin{equation*}
    \abs{D(H_{\frac{p}{2}}(\abs{\nabla u_p}))}^2 \, \approx \, (\abs{\nabla u_p} -1)_+^{p-2} \abs{\nabla^2 u_p}^2.
\end{equation*}
And we can effortlessly attain
\begin{equation}\label{Du-1(p-2)}
\begin{aligned}
    (\abs{\nabla u_p} -1)_+^{p-2} &= \left( (\frac{\abs{x}}{N} f^{**}(C_N \abs{x}^N))^{\frac{1}{p-1}} +1 -1 \right)^{p-2}=\\
    &=\left(\frac{\abs{x}}{N} f^{**}(C_N \abs{x}^N) \right)^{\frac{p-2}{p-1}} 
    \end{aligned}
\end{equation}
and
\begin{equation}\label{D2u2}
    \abs{\nabla ^2 u_p}^2 \approx \abs{x}^{2\frac{2-p}{p-1}} \left(  f^{**}(C_N \abs{x}^N)\right)^{\frac{2}{p-1}} + \frac{1}{\abs{x}^2}.
\end{equation}
In the case $p \geq 2$  using \eqref{stimeperlederivate} in combination with \eqref{Du-1(p-2)} and \eqref{D2u2}, we get
\begin{align}
    \int_{B_R} (\abs{\nabla u_p} -1)_+^{p-2} \abs{\nabla^2 u_p}^2 dx &\leq c \, \int_{B_R} \Bigg[\left(\abs{x} f^{**}(C_N \abs{x}^N) \right)^{\frac{p-2}{p-1}} \left( \abs{x}^{2\frac{2-p}{p-1}} \left(  f^{**}(C_N \abs{x}^N)\right)^{\frac{2}{p-1}} + \frac{1}{\abs{x}^2}  \right) dx \notag\\
   \notag \\
   & =  c \int_{B_R} \Bigg[ \left( \abs{x}^{\frac{2-p}{p-1}} \right) \left(  f^{**}(C_N \abs{x}^N)\right)^{\frac{p}{p-1}} + \left( \abs{x}^{\frac{-p}{p-1}} \right) \left(  f^{**}(C_N \abs{x}^N)\right)^{\frac{p-2}{p-1}} \Bigg] dx \notag\\
   \notag \\
     & \leq  c \int_{B_R} \Bigg[ \left( \abs{x}^{\frac{2-p}{p-1}} \right) \left(  \abs{x}^{ -\frac{N}{r}}\right)^{\frac{p}{p-1}} + \left( \abs{x}^{\frac{-p}{p-1}} \right) \left(  \abs{x}^{-\frac{N}{r}})\right)^{\frac{p-2}{p-1}} \Bigg] dx \notag\\
   \notag \\
   & =c \int_{B_R} \Bigg[  \abs{x}^{\frac{2-p}{p-1}- \frac{N}{r}(\frac{p}{p-1})}  + \abs{x}^{\frac{-p}{p-1}- \frac{N}{r}(\frac{p-2}{p-1})}  \Bigg] dx, \notag
    \end{align}
    where we used \eqref{stimeperlederivate}.\\
Therefore to have $ H_{\frac{p}{2}}(\abs{\nabla u_p}) \in W^{1,2}(B_R)$, it suffices to satisfy the following conditions 
\begin{equation}\label{condHp/2}
    \begin{cases}
      \frac{2-p}{p-1}- \frac{N}{r}(\frac{p}{p-1}) +N >0\\
    \frac{-p}{p-1}- \frac{N}{r}(\frac{p-2}{p-1}) +N >0 .
    \end{cases}
\end{equation}
And since
\begin{equation*}
     \frac{-p}{p-1}- \frac{N}{r} \left( \frac{p-2}{p-1} \right) > \frac{2-p}{p-1}- \frac{N}{r} \left( \frac{p}{p-1} \right),
\end{equation*}
it suffices to satisfy only the first condition in \eqref{condHp/2}, i.e.
\begin{equation}\label{rmaggHpmezzi}
r > \frac{Np}{N(p-1) +2-p}.
\end{equation}
This concludes the proof.
 \end{proof}
\noindent Subsequently we can observe that if $q < \frac{Np}{N+p-2}$ then
 \begin{equation*}
    \frac{Np}{N(p-1)+ (2-p)} >  \frac{Nq}{N(p-1)+ q(2-p)}
 \end{equation*}
 and so, the exponents $r$ satisfying \eqref{rmaggHpmezzi} automatically satisfy also \eqref{rmagg}. Then, for $p \geq 2$, we acquire that if
 $$  \,f \in L^{r, \infty}(B_R) \, \text{  with } \, r>\frac{Np}{N(p-1)+ (2-p)}$$ then 
 $$H_{\frac{p}{2}}(\nabla u_p) \in W^{1,2}(B_R)$$ and $$\nabla u_p \in W^{1,q}(B_R), \quad \forall  \,  q < \frac{Np}{N+p-2}.$$
 \\
 \noindent In particular $ \nabla u_p \in W^{1,2}(B_R)$ if $$\frac{Np}{N+p-2}>2,$$
 i.e. $p>2$ and $N>2$.
 \begin{singlespace}
\section{Some stability estimate as $p \to 1$}\label{sec:Stability}
\end{singlespace}
As a consequence of Theorem \ref{Teo1}, we can analyze the behaviour of the family $(\nabla u_p)_p$ as $p \to 1$. We start with
the following
\begin{prop}\label{prop4.1}
    Let $u_p$ be a solution to \eqref{1}, then denoting by
    \begin{equation*}\label{campoz}
        z = - \left( \frac{f^{**}(C_N \abs{x}^N)}{N}\right) \, x,
    \end{equation*}
    we have
    \begin{equation*}\label{z}
         \left(  \abs{\nabla u_{p} }-1 \right)_{+}^{p-1} \, \displaystyle{\frac{\nabla u_{p}}{\abs{\nabla u_{p}}}} \, = \, z,
    \end{equation*}
   for every $p>1$.
\end{prop}
\begin{proof}
For the solution $u_p$ of \eqref{1} at \eqref{NuovaSo}, a direct calculation shows that
  
    \begin{align}
         \left(  \abs{\nabla u_{p} }-1 \right)_{+}^{p-1} \, \displaystyle{\frac{\nabla u_{p}}{\abs{\nabla u_{p}}}}&= -\left[\, \left( \frac{\abs{x}}{N}f^{**}(C_N \abs{x}^N)\right)^{\frac{1}{p-1}}\,\right]^{p-1}\, \frac{x}{\abs{x}}  \notag \\
        \notag \\
         &= -\left( \frac{\abs{x}}{N}f^{**}(C_N \abs{x}^N)\right) \,\frac{x}{\abs{x}} \notag \\
        \notag \\
         &=  - \left( \frac{f^{**}(C_N \abs{x}^N)}{N}\right) \, x \,= \, z. \notag
         \end{align}
It is worth noting that
  \begin{equation*}
    -  \mathrm{div} \left(  ( \abs{\nabla u_p} -1 )_{+}^{p-1}\, \frac{\nabla u_p}{\abs{\nabla u_p}}   \right) = f,
  \end{equation*}
i.e. $$-\mathrm{div}(z)=f.$$
\end{proof}

\noindent We can observe that these results extend to the case of widely degenerate \textit{p}-Laplace type equations those in \cite[Proposition 3.1]{Merc}; therefore, we can obtain the equivalent theorems as in \cite[Paragraph 3]{Merc}. In particular, we can acquire Theorem \ref{TeoMerc}.\\

\section{Example}\label{sec:Example}
In this section, using the example given in \cite[5.An Example]{BraEsempio}  we show that Theorems \ref{TeoD2} and \ref{TeoHmezzi} and Corollary \ref{Cor1} are sharp.\\
For $\beta>1$, let us consider 
$$f(x)= \abs{x}^{- \beta}$$
and since
$$f^{**}(C_N \abs{x}^N) \approx \abs{x}^{- \beta}$$
then we can compute
$$(\abs{\nabla u} -1)_+^{p-2} \abs{\nabla^2 u}^2 \approx  \abs{x}^{\frac{2-p}{p-1} -\frac{ \beta p}{p-1}} +  \abs{x}^{\frac{-p}{p-1}+\frac{- \beta(p-2)}{p-1}} \approx \abs{x}^{\frac{2-p- \beta p}{p-1} }+  \abs{x}^{\frac{2 \beta - \beta p -p}{p-1}} \approx \abs{x}^{\frac{2-p- \beta p}{p-1} } $$
In order to have $(\abs{\nabla u} -1)_+^{p-2} \abs{\nabla^2 u}^2 \in L^1(B_R)$, the following condition must holds
\begin{align}
& \frac{2-p- \beta p}{p-1} + N >0 \qquad \text{   that is  } \notag \\
    &\beta < N-1 - \frac{N-2}{p} \quad \text{   and then } \notag \\
    &\frac{1}{\beta} > \frac{p}{Np- N +2 - p} \notag\\
    &\frac{N}{\beta} > \frac{Np}{N(p-1) +2 - p} \label{Beta}
\end{align}
 
Therefore
$$f \in L^{\frac{N}{\beta}, \infty}(B_R) \Rightarrow (\abs{\nabla u} -1)_+^{p-2} \abs{\nabla^2 u}^2 \in L^{1}(B_R)$$
for every $\beta >0$ such that $ \frac{N}{\beta} > \frac{Np}{N(p-1) +2 - p}, \quad \forall p> 1. $
Obviously for ${\hat{\beta}} >\beta$, since $$\frac{N}{\hat{\beta}} < \frac{N}{\beta},$$ \eqref{Beta} cannot be satisfied
and so
$$f \in L^{\frac{N}{\hat{\beta}}, \infty}(B_R) \text{   but   } (\abs{\nabla u} -1)_+^{p-2} \abs{\nabla^2 u}^2 \notin L^{1}(B_R)$$\\
Hence Theorem \ref{TeoD2} is sharp.\\
Now we calculate 
\begin{equation*}
  \abs{ \nabla^2 u}^q \approx     \abs{x} ^{\frac{2-p}{p-1}q +{\frac{- \beta q}{p-1}}} + \abs{x}^{-q} \approx \abs{x} ^{\frac{2-p}{p-1}q+ \frac{- \beta q}{p-1}}
\end{equation*}
To order to have $\abs{ \nabla^2 u}^q \in L^q(B_R)$, the following must be true
\begin{align}
   & \frac{q(2-p- \beta)}{p-1} +N >0,  \quad \text{   that is  } \notag \\
    &\beta < \frac{N(p-1)+ q (2-p)}{q}, \quad \text{    then} \notag \\
    &\frac{N}{\beta} > \frac{Nq}{N(p-1)+ q (2-p)} \notag
\end{align}
that is, Theorem \ref{TeoD2} is sharp.\\
At this point we can deduce
\begin{equation*}
  \abs{ \nabla u} \approx     \abs{x} ^{\frac{1-\beta}{p-1}}.
\end{equation*}
To order to have $\abs{ \nabla u} \in L^{\infty}(B_R)$, we have to have $\beta\leq 1$,  confirming the sharpness of the Corollary \ref{Cor1}.\\
Now, in order to compare our results with \cite[Theorem 1.1] {BraEsempio}
let us choose as in \cite{BraEsempio}
$$\beta = (\alpha +1)(p-1)+1$$
then \eqref{Beta} becomes
\begin{align}
    &(\alpha +1)(p-1)+1 < \frac{2-p}{p} + N - \frac{N}{p} \notag \\
    & (\alpha +1)(p-1) < 2(\frac{1-p}{p}) + N (\frac{p-1}{p}) \notag \\
    & \alpha + 1 < - \frac{2}{p} + \frac{N}{p} \notag \\
    & \alpha < \frac{N-2}{p} - 1 = \hat{\alpha}\notag
\end{align}
We'd like to explicitly mention that this is the bound found in \cite{BraEsempio}. Here, it's a summability threshold for the datum, while in \cite{BraEsempio} it was a threshold for its fractional differentiability.
\vspace{20mm}
\textbf{Acknowledgements}\\
The author is member of GNAMPA (Gruppo Nazionale per l'Analisi Matematica, la Probabilità e le loro Applicazioni) of INdAM (Istituto Nazionale di Alta Matematica)

\begin{singlespace}
\lyxaddress{\noindent \textbf{$\quad$}\\
$\hspace*{1em}$\textbf{Stefania Russo}\\
Dipartimento di Matematica e Applicazioni ``R. Caccioppoli''\\
Università degli Studi di Napoli ``Federico II''\\
Via Cintia, 80126 Napoli, Italy.\\
\textit{E-mail address}: stefania.russo3@unina.it}
\end{singlespace}

\end{document}